\documentclass[12pt]{amsart}

\usepackage[T1]{fontenc}
\usepackage[utf8]{inputenc}
\usepackage{amssymb}
\usepackage[all]{xy}
\usepackage{amscd,verbatim}
\usepackage{enumitem}

\usepackage{color}
\usepackage[colorlinks,linkcolor=blue,citecolor=blue,urlcolor=red]{hyperref}

\newcommand{\End}{\operatorname{End}}

\newcommand{\Id}{\operatorname{Id}}
\newcommand{\adj}{\operatorname{adj}}

\renewcommand{\lim}{\varprojlim}
\newcommand{\colim}{\varinjlim}

\newcommand{\sA}{\mathcal{A}}

\newcommand{\sC}{\mathcal{C}}

\newcommand{\A}{\mathbf{A}}

\newcommand{\D}{\mathbf{D}}

\newcommand{\M}{\mathbf{M}}

\newcommand{\pre}{{\operatorname{pre}}}
\newcommand{\ass}{{\operatorname{ass}}}

\newcommand{\Ens}{\operatorname{\mathbf{Set}}}

\newcommand{\Rep}{\operatorname{\mathbf{Rep}}}

\newcommand{\Cat}{\operatorname{\mathbb{CAT}}}

\newcommand{\IM}{\operatorname{Im}}

\renewcommand{\epsilon}{\varepsilon}
\renewcommand{\phi}{\varphi}

\newcommand{\iso}{\overset{\sim}{\longrightarrow}}
\newcommand{\by}{\xrightarrow}

\newtheorem{thm}{Theorem}[subsection]
\newtheorem{prop}[thm]{Proposition}
\newtheorem{cor}[thm]{Corollary}
\newtheorem{lemme}[thm]{Lemma}
\theoremstyle{definition}
\newtheorem{defn}[thm]{Definition}
\newtheorem{hyp}[thm]{Hypothesis}
\theoremstyle{remark}
\newtheorem{rque}[thm]{Remark}
\newtheorem{rques}[thm]{Remarks}

\newtheorem{ex}[thm]{Example}
\newtheorem{exs}[thm]{Examples}
\newtheorem{excan}[thm]{Canonical example}

\newcounter{spec}
\newenvironment{thlist}{\begin{list}{\rm{(\roman{spec})}}%
{\usecounter{spec}\labelwidth=20pt\itemindent=0pt\labelsep=10pt}}%
{\end{list}}%
\numberwithin{equation}{subsection}
\begin{document}
\title{On the Bénabou-Roubaud theorem }
\author{Bruno Kahn}
\address{CNRS, Sorbonne Université and Université Paris Cité, IMJ-PRG\\ Case 247\\4 place
Jussieu\\75252 Paris Cedex 05\\France}
\email{bruno.kahn@imj-prg.fr}
\date{January 17, 2025}
\begin{abstract}
We give a detailed proof of the Bénabou-Roubaud theorem. As a byproduct, it yields a weakening of its hypotheses: the base category does not need fibre products and the Beck-Chevalley condition, in the form of a natural transformation, can be weakened by only requiring the latter to be epi.
\end{abstract}
\maketitle

\hfill To the memory of Jacques Roubaud.

\enlargethispage*{20pt}

\subsection*{Introduction} The Bénabou-Roubaud theorem \cite{br} establishes, under certain conditions, an equivalence of categories between a category of descent data and a category of algebras over a monad. This result is widely cited, but \cite{br} is a note ``without proofs'' and the ones I know in the literature are a bit terse (\cite[pp. 50/51]{guo}, \cite[proof of Lemma 4.1]{herminda}, \cite[Th. 8.5]{ln}), \cite[3.7]{facets}; moreover,  \cite{herminda} and \cite{ln} are formulated in more general contexts.

The aim of this note is to provide a detailed proof of this theorem in its original context. This exegesis has the advantage of showing that the original hypotheses can be weakened: it is not necessary to suppose that the base category admits fibred products\footnote{As was pointed out by the referee, the corresponding arguments are related to Street's notion of descent object relative to a truncated (co)simplicial category as in the beginning of \cite{street}; but a ``truncated cyclic category'' à la Connes is also lurking in Proposition \ref{p4.1} b).}, and the Chevalley property of \cite{br}, formulated as an exchange condition, can also be weakened by requiring that the base change morphisms be only epi.  I hope this will be useful to some readers. I also provided a proof of the equivalence between Chevalley's property and the exchange condition (attributed to Beck, but see remark \ref{r1}): this result is part of the folklore but, here again, I had difficulty finding a published proof. In Corollary \ref{c2}, I give a condition (probably too strong) for the Eilenberg-Moore comparison functor to be essentially surjective. Finally, I give cases in Proposition \ref{p3} where the exchange isomorphism holds; this is certainly classical, but it recovers conceptually Mackey's formula for the induced representations of a group (Example \ref{ex3}).

\subsubsection*{Notation and conventions}
I keep that of \cite{br}: thus $P:\M\to \A$ is a bifibrant functor in the sense of \cite[\S 10]{sga1}. If $A\in \A$, we denote by $\M(A)$ the fibre of $P$ above $A$. For an arrow $a:A_1\to A_0$ of $\A$, we write $a^*:\M(A_0)\to \M(A_1)$ and $a_*:\M(A_1)\to \M(A_0)$ for the associated inverse and direct image functors ($a_*$ is \emph{left} adjoint to $a^*$) and $\eta^a$, $\epsilon^a$ for the associated unit and counit. We also write $T^a=a^*a_*$ for the associated monad, equipped with its unit $\eta^a$ and its multiplication $\mu^a=a^*\epsilon^aa_*$. \emph{We do not assume the existence of fibre products in $\A$}.

In order to simplify calculations, we shall assume that the pseudofunctor $a\mapsto a^*$ is a functor. This can be justified by the fact that it can be rectified; more precisely, the morphism of pseudofunctors $i\mapsto F_i$ of \cite[\S 3, p. 141]{jardine} is clearly faithful, hence any parallel arrows in its source which become equal in its target are already equal. (One could also use \cite[I, Th. 2.4.2 or 2.4.4]{giraud}.) Then one can also choose the left adjoints $a\mapsto a_*$ to form a functor \cite[IV.8, Th. 1]{mcl}, which we do.

\subsection{Adjoint chases}\label{s1}
To elucidate certain statements and proofs, I start by doing two things: 1) ``deploy'' the single object $M_1$ of \cite{br} into several, which will allow us to remove the quotation marks from  ``natural'' at the bottom of \cite[ p. 96]{br}, 2) not assume the Beck-Chevalley condition to begin with, which will allow us to clarify the functoriality in the first lemma of the note and to weaken hypotheses.

\subsubsection{}\label{s1.1} Let $a$ be as above; still following the notation of \cite{br}, we give ourselves a commutative square
\begin{equation}\label{eq6}
\begin{CD}
A_2@>a_2>>A_1\\
@V a_1 VV @V aVV\\
A_1@> a>> A_0.
\end{CD}
\end{equation}
except that we don't require it to be Cartesian. The equality $a_1^* a^*=a_2^*a^*$ 
yields a base change morphism
\begin{equation}\label{eq0}
\chi: (a_2)_*a_1^*\Rightarrow T^a 
\end{equation}
equal to the composition $\epsilon^{a_2}T^a \circ (a_2)_* a_1^*\eta^{a}$.
Hence a map
\begin{multline}\label{eq4}
\xi_{M,N}=\xi:\M(A_1)(T^a M,N)\by{\chi_{M}^*} \M(A_1)((a_2)_*a_1^*M,N)\\
\overset{\adj}{\iso} \M(A_2)(a_1^*M,a_2^*N) 
\end{multline}
for $M,N\in \M(A_1)$. It goes in the \emph{opposite} direction to the map $K^a$ of \cite{br}, which we will find back in \eqref{eq1.1}. (See also Remark \ref{eq4.3} in that section.)

\begin{rque}\label{r1} The morphism \eqref{eq0} is sometimes called ``Beck transformation''. However, it already appears in SGA4 (1963/64) to formulate the proper base change and smooth base change theorems \cite[§4]{sga4}. I have adopted the terminology ``base change morphism'' in reference to this seminar.
\end{rque}

\begin{lemme}[key lemma]\label{l6} For any $\phi\in \M(A_1)(T^a M,N)$, one has
\[\xi(\phi) =
a_2^*\phi \circ a_1^*\eta^{a}_{M}.\]
\end{lemme}

\begin{proof} For $\psi\in \M(A_1)(a_2)_*a_1^*M,N)$ one has $\adj(\psi)=a_2^*\psi \circ \eta^{a_2}_{a_1^*M}$, hence 
\begin{align*}
\xi(\phi)= \adj(\phi \circ\chi_{M}) &=a_2^*(\phi \circ\chi_{M}) \circ \eta^{a_2}_{a_1^*M}\\
&=a_2^*(\phi \circ(\epsilon^{a_2}T^a \circ (a_2)_*a_1^*\eta^{a})_{M}) \circ \eta^{a_2}_{a_1^*M}\\
&=a_2^*\phi \circ a_2^*\epsilon^{a_2}_{T^a M}\circ a_2^*(a_2)_* a_1^*\eta^{a}_{M} \circ \eta^{a_2}_{a_1^*M}\\
&=a_2^*\phi \circ a_2^*\epsilon^{a_2}_{T^a M}\circ  \eta^{a_2}_{a_1^*T^a M}\circ a_1^*\eta^{a}_{M}\\
&=a_2^*\phi\circ a_1^*\eta^{a}_{M}
\end{align*}
where we successively used the naturality of $\eta^{a_2}$ and an adjunction identity.\end{proof}

\subsubsection{}\label{s1.2} Let $A_3\in \A$ be equipped with ``projections'' $p_1,p_2,p_3:A_3\to A_2$. We assume that the ``face identities'' $a_1 p_2=a_1 p_3$, $a_1p_1=a_2p_3$, $a_2p_1=a_2p_2$ are satisfied; we call these morphisms respectively $b_1,b_2,b_3$.

\begin{excan}\label{ex5} $A_2 = A_1\times_{A_0} A_1$, $A_3=A_1\times_{A_0} A_1\times_{A_0} A_1$, all morphisms given by the natural projections.
\end{excan}

We then have maps, for  $i<j$
\begin{equation}\label{eq1.3}
\alpha_{ij}(M,N)=\alpha_{ij}:\M(A_2)(a_1^*M,a_2^*N)\to \M(A_3)(b_i^*M,b_j^*N)
\end{equation}
given by
\[\alpha_{12}=p_3^*,\quad \alpha_{13}=p_2^*, \quad \alpha_{23}=p_1^*\]
hence composite maps
\begin{equation}\label{eq1}
\theta_{ij}=\alpha_{ij}\circ \xi:\M(A_1)(T^a M,N)\to \M(A)(b_i^*M,b_j^*N).
\end{equation}
 
In addition, we have the multiplication of $T^a$ mentioned in the notations:
\begin{equation}\label{eq3}
\mu^a=a^*\epsilon^{a} a_*: T^a T^a  \Rightarrow T^a .
\end{equation}

The commutative square\footnote{Note that it is Cartesian in the canonical example.}
\begin{equation}\label{eq1.2}
\begin{CD}
A_3@>p_3>> A_2\\
@Vp_1 VV @Va_2 VV\\
A_2@>a_1>> A_1
\end{CD}
\end{equation}
yields another base change morphism $\lambda:(p_1 )_*p_3^*\Rightarrow a_1^*(a_2)_*$, hence a composition
\begin{equation}\label{eq8}
(b_3)_*b_1^*= (a_2)_*(p_1)_*p_3^*a_1^*\overset{(a_2)_*\lambda a_1^*}{\Longrightarrow} (a_2)_*a_1^*(a_2)_*a_1^* \overset{\chi*\chi}{\Longrightarrow} T^a T^a 
\end{equation}
which, together with adjunction, induces a map
\begin{equation}\label{eq1a}
\rho:\M(A_1)(T^a T^a M,N)\to \M(A_2)(b_1^*M,b_3^*N).
\end{equation}

\begin{lemme}\label{l1} a) The diagram of natural transformations
\[\xymatrix{
(a_2)_*(p_1)_*p_3^*a_1^*\ar@{=>}[d]_{(a_2)_*\lambda a_1^*}&(b_3)_*b_1^*\ar@{=}[l]\ar@{=}[r]& (a_2)_*(p_2)_*p_2^*a_1^*\ar@{=>}[d]^{(a_2)_*\epsilon^{p_2}a_1^*}\\
(a_2)_*a_1^*(a_2)_*a_1^*\ar@{=>}[d]_{\chi*\chi}&& (a_2)_*a_1^*\ar@{=>}[d]^\chi\\
T^aT^a\ar@{=>}[rr]^{\mu^a}&& T^a
}\]
is commutative.\\
b) One has $\theta_{13}=\rho \circ \mu_a^*$ (see \eqref{eq1}, \eqref{eq3} and \eqref{eq1a}).
\end{lemme}

\begin{proof} 
a) is a matter of developing the base change morphisms as done for $\chi$ just below \eqref{eq0} (see proof of Lemma \ref{l6}). This yields a commutative diagram
\[\begin{CD}
\M(A_1)((b_3)_*b_1^* M,N)@<((a_2)_*\epsilon^{p_2}a_1^*)^*<< \M(A_1)((a_2)_*a_1^* M,N)\\
@A\eqref{eq8}^*AA @A\chi^* AA\\
\M(A_1)(T^a T^a M,N)@<(\mu^a)^*<< \M(A_1)(T^a M,N)
\end{CD}\]
from which we get b) by developing the adjunction isomorphism for $((b_3)_*,b_3^*)$.
\end{proof}

Let now $M_1,M_2,M_3\in \M(A_1)$ and $\phi_{ij}\in \M(A_1)(T^a M_i,M_j)$ be three morphisms. We have a not necessarily commutative square:
\begin{equation}\label{eq2}
\begin{CD}
T^a T^a M_1@>T^a \phi_{12}>> T^a M_2\\
@V(\mu_a)_{M_1} VV @V\phi_{23} VV\\
T^a M_1@>\phi_{13}>> M_3.
\end{CD}
\end{equation}

Write $\hat{\phi}_{ij}=\theta_{ij}(\phi_{ij}):b_i^*M_i\to b_j^*M_j$. 

\begin{lemme}\label{l2} Let $\psi$ (resp. $\psi'$) be the composition of \eqref{eq2} passing through $T^a M_2$ (resp. through $T^a M_1$ ). Then $\rho(\psi)= \hat{\phi}_{23} \circ \hat{\phi}_{12}$ and  $\rho(\psi')=\hat{\phi}_{13}$.
\end{lemme}

\begin{proof} The first point follows from a standard adjunction calculation similar to the previous ones, and the second follows from lemma \ref{l1}.
\end{proof}

\begin{prop}\label{p1} If \eqref{eq2} commutes, we have $\hat{\phi}_{13}=\hat{\phi}_{23} \circ \hat{\phi}_{12}$; the converse is true if $\rho$ is injective in \eqref{eq1a}.
\end{prop}

\begin{proof} This is obvious in view of Lemma \ref{l2}.
\end{proof}

In \eqref{eq4}, assume that $M=N$ is of the form $a^*M_0$ and write $p=aa_1=aa_2:A_2\to A_0$. We have a composition
\begin{multline}\label{eq4a}
\M(A_1)(M,a^*M_0)\iso \M(A_0)(a_*M,M_0)\\
\by{a^*}\M(A_1)(T^a M,a^*M_0)\by{\xi} \M(A_2)(a_1^*M,p^*M_0) 
\end{multline}
where the first arrow is the adjunction isomorphism. A new adjoint chase gives:

\begin{lemme}\label{l3} The composition \eqref{eq4a} is induced by $a_1^*$.\qed
\end{lemme}

\subsection{Exchange condition and weak exchange condition} Now we introduce the

\begin{defn}\label{d1} A commutative square \eqref{eq6} is said to satisfy the \emph{exchange condition} if the base change morphism \eqref{eq0} is an isomorphism; we say that \eqref{eq6} satisfies the \emph{weak exchange condition} if \eqref{eq0} is epi.
\end{defn}

\begin{lemme}[cf. {\cite[Prop. 11]{pav}} and {\cite[II.3]{pav2}}]\label{l5}   The exchange condition of Definition \ref{d1} is equivalent to the Chevalley condition (C) of \cite{br}.
\end{lemme}

\begin{proof} Recall this condition: given a commutative square 
\begin{equation}\label{eq7}
\begin{CD}
M'_1@>k_1>> M_1\\
@V\chi' VV @V\chi VV\\ 
M'_0@>k_0 >> M_0,
\end{CD}
\end{equation}
above \eqref{eq6} (where we take $(i,j)=(1,2)$ to fix ideas), if $\chi$ and $\chi'$ are Cartesian and $k_0 $ is co-Cartesian, then $k_1$ is co-Cartesian.

I will show that the exchange condition is equivalent to each of the following two conditions: (C) and
\begin{description}
 \item[(C')]  if $k_0$ and $k_1$ are co-Cartesian and $\chi'$ is Cartesian, then $\chi$ is Cartesian.
\end{description}

Let us translate the commutativity of \eqref{eq7} in terms of the square
\begin{equation}\label{eq5}
\begin{CD}
\M(A_2)@ >(a_2)_*>>\M(A_1)\\
@Aa_1^*AA @A a^* AA\\
\M(A_1)@>a_*>> \M(A_0).
\end{CD}
\end{equation}

The morphisms of \eqref{eq7} correspond to morphisms $\tilde k_0:a_*M'_0\allowbreak\to M_0$, $\tilde k_1:(a_1)_*M'_1\to M_1$, $\tilde \chi: M_1\to a^*M_0$ and $\tilde \chi': M'_1\to a_2^*M'_0$, which fit in a commutative diagram of $\M(A_1)$:
\[\xymatrix{
(a_2)_*a_1^* M'_0\ar[rr]^c&& T^a  M'_0\ar[d]^{a^*\tilde k_0}\\
(a_2)_* M'_1\ar[u]^{(a_2)_*\tilde \chi'}\ar[r]^{\tilde k_1} & M_1\ar[r]^{\tilde \chi}& a^* M_0
}\]
where $c$ is the base change morphism of \eqref{eq0}. The cartesianity conditions on $\chi$ and $\chi'$ (resp. co-cartesianity conditions on $k_0$ and $k_1$) amount to requesting the corresponding morphisms decorated with a $\tilde{}$ to be isomorphisms.

Suppose $c$ is an isomorphism. If $\tilde \chi'$ and $\tilde k_0$ are isomorphisms, $\tilde \chi$ is an isomorphism if and only if $\tilde k_1$ is. Thus, the exchange condition implies conditions (C) and (C'). Conversely, $M'_0$ being given, let $\tilde k_0$, $\tilde \chi$ and $\tilde \chi'$ be identities, which successively defines $M_0$, $M_1$ and $M' _1$. The arrow $c$ then defines an arrow $\tilde k_1$, which is an isomorphism if and only if so is $c$. This shows that the exchange condition is implied by (C), and we argue symmetrically for (C') by taking $\tilde \chi'$, $\tilde k_1$ and $\tilde k_0$ to be identities.
\end{proof}

\enlargethispage*{30pt}

\begin{rques} a) This proof did not use the hypothesis that \eqref{eq6} be Cartesian.\\
b) Under conservativity assumptions for $a_2^*$ or $a^*$, we obtain converses to (C) and (C').
\end{rques}

\subsection{Pre-descent data}\label{s3} Here we come back to the set-up of Section \ref{s1}: namely, we give ourselves a commutative diagram \eqref{eq6} as in \S\ref{s1.1} and a system $(A_3,p_1,p_2,p_3)$ as in the beginning of \S\ref{s1.2} satisfying the identities of \emph{loc. cit.} In other words, we have a set of objects and morphisms of $\A$ 
\[(A_0,A_1,A_2,A_3,a,a_1,a_2,p_1,p_2,p_3)\] 
subject to the  relations 
\[aa_1=aa_2, \quad a_1 p_2=a_1 p_3, \quad a_1p_1=a_2p_3, \quad a_2p_1=a_2p_2.\]

Let $M\in \M(A_1)$ and  $v\in\M(A_2)(a_1^*M,a_2^* M)$. We associate to $v$ three morphisms
\[\hat{\phi}_ {ij}=\alpha_{ij}(v):b_i^*M\to b_j^*M\quad (i<j) \]
where $\alpha_{ij}$ are the maps of \eqref{eq1.3}.

\begin{defn}\label{d2} We say that $v$ is a \emph{pre-descent datum} on $M$ if the $\hat{\phi}_ {ij}$ satisfy the condition $\hat{\phi}_{13}=\hat{\phi}_{23} \circ \hat{\phi}_{12}$ of Proposition \ref{p1}.
We write $\D^\pre$ for the category whose objects are pairs $(M,v)$, where $v$ is a pre-descent datum on $M$, and whose morphisms are those of $\M(A_1)$ which commute with pre-descent data.
\end{defn}

Let us introduce the 

\begin{hyp}\label{h1} The weak exchange condition is verified by the squares \eqref{eq6} and \eqref{eq1.2}.
\end{hyp}

\begin{prop}[cf. \protect{\cite[lemme]{br}}]\label{c1}
In \eqref{eq2}, assume $\phi_{12}=\phi_{23}=\phi_{13}=:\phi$. If $\phi$ satisfies the associativity condition of a $T^a$-algebra, then $\xi(\phi)$ in \eqref{eq4} is a pre-descent datum; the converse is true under Hypothesis \ref{h1}.\end{prop}

\begin{proof} In view of Proposition \ref{p1}, it suffices to show that Hypothesis  \ref{h1} implies the injectivity of $\rho$, which is induced by the composition of the two natural transformations of \eqref{eq8}. The second is epi, therefore induces an injection on  Hom's, and so does the first by adjunction.
\end{proof}

\begin{cor}\label{c1a} Let $\M^a_\ass$ denote the category of associative $T^a$-algebras which are not necessarily unital. Then Proposition \ref{c1} defines a faithful functor $\xi:\M^a_\ass\to \D^\pre$ commuting with the forgetful functors to $\M(A_1)$; under Hypothesis \ref{h1}, it is an isomorphism of categories. 
\end{cor}

\begin{proof} Commutation of $\xi$ with the forgetful functors is obvious. This already shows that it is faithful; under Hypothesis  \ref{h1}, it is essentially surjective by Proposition \ref{c1} and we see immediately that it is also full.
\end{proof}

\subsection{The unit condition}\label{s4} We keep the hypotheses and notation of Section \ref{s3}, and  introduce an additional ingredient: a ``diagonal'' morphism $\Delta:A_1\to A_2$ such that $a_1\Delta=a_2 \Delta=1_{A_1}$.

\begin{defn}\label{d3} A \emph{descent datum} on $M$ is a pre-descent datum $v$ such that $\Delta^* v=1_M$. We denote by $\D$ the full subcategory of $\D^\pre$ given by the descent data. 
\end{defn}

Let $\M^a\subset\M^a_\ass$ be the category of $T^a$-algebras. 

\begin{thm}[cf. \protect{\cite[théorème]{br}}]\label{t1}  For all $\phi\in \M(A_1)(T^a M,M)$, we have
\begin{equation}\label{eq4.1}
\Delta^*\xi(\phi) = \phi\circ \eta^a_M.
\end{equation}
In particular, $\xi(\M^a)\subset \D$ and $\xi:\M^a\to \D$ is an isomorphism of categories under Hypothesis \ref{h1}.
\end{thm}

\begin{proof}  Suppose  that $M=N$ in Lemma \ref{l6}. Applying $\Delta^*$ to its identity, we get
\eqref{eq4.1}. In particular, if $\phi$ is the action of a $T^a$-algebra then $v=\xi(\phi)$ verifies $\Delta^* v=1_M$. We conclude with Corollary \ref{c1a}. 
\end{proof}

As in \cite[VI.3, Th. 1]{mcl}, we have the Eilenberg-Moore comparison functor
\begin{align}\label{eq1.1}
K^a:\M(A_0)&\to \M^a \\
M_0&\mapsto (a^*M_0,a^* \epsilon^a_{M_0}).\notag
\end{align}

Lemma \ref{l3} yields:

\begin{prop}\label{p2} We have $\xi(a^*\epsilon^a_{M_0})=1_{M_0}$. In other words, in the diagram
\[\xymatrix{
\M(A_0)\ar[r]^{\Psi^a}\ar[rd]_{K^a}&\D\ar[r]^{U^a}&\M(A_1)\\
&\M^a\ar[u]^\xi\ar[ur]_{U^{T^a}}
}\]
the left triangle commutes (as well as the right one, trivially). \qed
\end{prop}

\begin{rque}\label{eq4.3} In \cite[3.7]{facets}, Janelidze and Tholen construct a functor from $\D$ to $\M$ (same direction as in \cite{br}) by using the \emph{inverses} of the base change morphisms \eqref{eq0}.
\end{rque}

\begin{rque} In the canonical example \ref{ex5}, a pre-descent datum $v$ satisfies the condition of Definition \ref{d3} if and only if it is invertible (therefore is a descent datum in the classical sense): this follows from \cite[A.1.d pp. 303--304]{gro}. In \emph{loc. cit.}, Grothendieck uses an elegant Yoneda argument. It is an issue to see how this result extends to our more general situation: this is done in the next proposition. I am indebted to the referee for prodding me to investigate this. 

Note that I merely looked for what is necessary to translate Grothen\-dieck's arguments, and not for the greatest generality.
\end{rque}

\begin{prop}\label{p4.1} Let $(A_0,A_1,A_2,A_3,a,a_1,a_2,p_1,p_2,p_3)$ be as in Section \ref{s3}. Let $M\in \M(A_1)$ and let $v\in\M(A_2)(a_1^*M,a_2^* M)$ be a pre-descent datum as in Definition \ref{d2}. Further, let $\Delta$ be as in the beginning of the present section. Consider the following conditions:
\begin{thlist}
\item  $\Delta^*v=1_M$ (i.e. $v$ is a descent datum).
\item  $v$ is invertible.
\end{thlist}
Then:\\
a) $(ii) \Rightarrow (i)$ under one of the following conditions: there exists a morphism $s_1$ (resp. $s_2$) from $A_2$ to $A_3$ such that 
\[p_1s_1= \Delta a_2,\quad p_2s_1=p_3s_1=1\]
(resp.
\[p_1s_2=  p_2s_2=1, \quad p_3s_2=\Delta a_2).\]
b) $(i) \Rightarrow (ii)$ under the following condition: there exists an involution $\sigma$ of $A_2$ 
and a morphism $\Gamma:A_2\to A_3$ such that 
\[p_1\Gamma= \sigma,\quad p_2\Gamma = \Delta a_1,\quad p_3\Gamma = 1_{A_2}.\] 
\end{prop}

(In the case of the canonical example \ref{ex5}, we may take for $s_1$ and $s_2$ the partial diagonals, for $\sigma$ the exchange of factors and for $\Gamma$ the graph of $a_1$, given in formula by $(\alpha_1,\alpha_2)\mapsto (\alpha_1,\alpha_2,\alpha_1)$.)

\begin{proof} The predescent condition on $v$ is
\begin{equation}\label{eq4.2}
p_2^*v=p_1^*v\circ p_3^* v.
\end{equation}

a) Applying $s_1^*$ to \eqref{eq4.2}, we get
\[v=a_2^*\Delta^*v\circ v\]
hence $a_2^*\Delta^*v=1_{A_2}$ and
\[\Delta^*v=\Delta^*a_2^*\Delta^*v=1_{A_1}.\]

Same reasoning with $s_2$, \emph{mutatis mutandis}. Note that with $s_1$ (\emph{resp}. $s_2$), it suffices to assume that $v$ is right (\emph{resp}. left) cancellable.

b) Applying $\Gamma^*$ to \eqref{eq4.2}, we get
\[1_{A_2} =a_1^*\Delta^* v = \sigma^* v\circ v. \]

Applying now $\sigma^*$, we also get $v\circ \sigma^* v = 1_{A_2}$.
\end{proof}

\subsection{A supplement} Recall \cite[Ex. 8.7.8]{sga4.1} that a category is called \emph{Karoubian} if any idempotent endomorphism has an image.

\begin{prop}\label{p4} Let $a^*$ be fully faithful and $\M(A_0)$ Karoubian. Let $\phi:T^a M\to M$ satisfy the identity $\phi\circ \eta^a_{M}=1_{M}$. Then there exists $M_0\in \M(A_0)$ and an isomorphism $\nu:M\iso a^*M_0$ such that $\phi=\nu^{-1}\circ a^*\epsilon^a_{M_0 }\circ  T^a \nu$.
\end{prop}

\begin{proof} Let $e$ denote the idempotent $\eta^a_{M}\phi\in \End_{\M(A_1)}(T^a M)$. By hypothesis, $e=a^*\tilde e$ where $\tilde e$ is an idempotent of $\End_{\M(A_0)}(a_*M)$, with image $M_0$. Then $a^*M_0$ is isomorphic to the image $M$ of $e$ via a morphism $\nu$ as in the statement, such that
\[\nu\circ \phi=a^*\pi,\quad a^*\iota\circ \nu = \eta^a_{M}\]
where $\iota\pi$ is the epi-mono factorization of $\tilde e$.

To finish, it is enough to see that $a^*\pi=a^*\epsilon^a_{M_0}\circ T^a \nu$.
But we also have
\[\eta^a_{a^*M_0}\circ \nu = T^a \nu \circ \eta^a_{M}=T^a \nu \circ a^*\iota\circ \nu\]
hence $\eta^a_{a^*M_0}=T^a \nu \circ a^*\iota$. This concludes the proof, since $\eta^a_{a^*M_0}\circ a^*\epsilon^a_{M_0}$ is the epi-mono factorisation of the idempotent of $\End(T^a a^*M_0 )$ with image $a^*M_0$.
\end{proof}

We thus obtain the following complement:

\begin{cor}\label{c2} Assume Hypothesis \ref{h1}, and also that $a^*$ is fully faithful and $\M(A_0)$ Karoubian. Then \\
a) every unital $T^a$-algebra is associative;\\
b) $K^a$ is essentially surjective.\qed
\end{cor}

Can one weaken the full faithfulness assumption in this corollary? The following lemma does not seem sufficient:

\begin{lemme}\label{l4} Let $M,N\in \M(A_1)$. Then the map 
\[a^*:\M(A_0)(a_*M,a_*N)\to \M(A_1)(T^a M,T^a N)\] 
has a retraction $r$ given by $r(f)=\epsilon^a_{a_*N}\circ a_*f\circ a_*\eta^a_{M}$. More generally, we have an identity of the form $r(a^*g \circ f) = g\circ r(f)$.  
\end{lemme}

\begin{proof} For $f:T^a M\to T^a N$ and $g:a_*N\to a_*P$, we have
\[r(a^*g\circ f) = \epsilon^a_{a_*P}\circ a_*a^*g\circ a_*f\circ a_*\eta^a_{M}=  g\circ\epsilon^a_{a_*N}\circ a_*f\circ a_*\eta^a_{M}=g\circ r(f).\]

Taking $f=1_{T^a M}$, we obtain that $r$ is a retraction.
\end{proof}

\subsection{Appendix: a case where the exchange condition is verified}

Let $\sA$ be a category. Take for $\A$ the category of presheaves of sets on $\sA$. Write $\int A$ for the category associated to $A\in \A$ by the Grothendieck construction \cite[\S 8]{sga1}. Recall its definition in this simple case: the objects of $\int A$ are pairs  $(X,a)$ where $X\in \sA$ and $a\in A(X)$, and a morphism from $(X,a)$ to $(Y,b)$ is a morphism $f\in \sA(X,Y)$ such that $A(f)(x)=y$.  

Let $\sC$ be another category. We take for $\M$ the fibred category of representations of $\A$ in $\sC$: for $A\in \A$, an object of $\M(A)$ is a functor from $\int A$ to $\sC$. For all $a\in\A(A_1, A_0)$ we have an obvious pull-back functor $a^*:\M(A_0)\to \M(A_1)$, which has a left adjoint $a_*$ (direct image) given by the usual colimit if $\sC$ is cocomplete. We can then ask whether the exchange condition is true for Cartesian squares of $\sA$. 

\begin{prop}\label{p3} This is the case if $\sC$ is the category of sets $\Ens$, and more generally if $\sC$ admits a forgetful functor $\Omega:\sC\to \Ens$ with a left adjoint $L$ such that $(L,\Omega)$ satisfies the conditions of Beck's theorem \cite[VI.7, Th. 1]{mcl}.
\end{prop}

\begin{proof}
First suppose $\sC=\Ens$; to verify that \eqref{eq0} is a natural isomorphism, it is enough to test it on representable functors. Consider Diagram \eqref{eq5} again. For $(c,\gamma)\in \int A_1$ and $(d,\delta)\in \int A_1$ (with $c,d\in \sA$ and $\gamma\in A_1(c)$ , $\delta\in A_1(d)$), we have
\begin{multline*}
T^a y(c,\gamma)(d,\delta)=a^*y(c,a(\gamma))(d,\delta)=y(c,a(\gamma))(d,a(\delta))\\
=\{\phi\in \sA(d,c)\mid  \phi^*a(\gamma)=a(\delta)\}
\end{multline*}
and
\begin{align*}
(a_2)_*a_1^*y(c,\gamma)(d,\delta)&=\colim_{(e,\eta)\in (d,\delta)\downarrow a_2}a_1^*y(c,\gamma)(e,\eta)\\
&=\colim_{(e,\eta)\in (d,\delta)\downarrow a_2}y(c,\gamma)(e,a_1(\eta))\\
&=\colim_{(e,\eta)\in (d,\delta)\downarrow a_2}\{\psi\in \sA(e,c)\mid \psi^*\gamma=a_1(\eta)\}.
\end{align*}

We have
\[ (d,\delta)\downarrow a_2 = \{(e,\eta,\eta_2,\theta)\in \sA\times A_1(e)\times_{A_0(e)} A_1(e)\times \sA(d,e)\mid  \theta^*\eta_2=\delta\}.\]

This category has the initial set $\{(d,\eta_1,\delta,1_d)\mid a(\eta_1)=a(\delta)\}$, so
\begin{multline*}
(a_2)_*a_1^*y(c,\gamma)(d,\delta)=\coprod_{\{(\eta_1\in A_1(d)\mid a(\eta_1)=a(\delta)\}}\{\phi\in \sA(d,c)\mid \phi^*\gamma=\eta_1\}\\
=\{\phi\in \sA(d,c)\mid a(\phi^*\gamma)=a(\delta)\}
\end{multline*}
and the map $(a_2)_*a_1^*y(c,\gamma)(d,\delta)\to (a_2)_*(a^{ 12})^*y(c,\gamma)(d,\delta)$ is clearly equal to the identity.

\emph{General case}: let us write more precisely $\M^\sC(A)=\Cat(\int A,\sC)$. The functors $L$ and $\Omega$ induce pairs of adjoint functors (same notation)
\[L:\M^{\Ens}(A)\leftrightarrows \M^\sC(A):\Omega.\]

These two functors commute with pull-backs; as $L$ is a left adjoint, it also commutes with direct images. Therefore, in the above situation, the base change morphism $\chi_M:(a_2)_*a_1^*M\to T^a M $ is an isomorphism when $M\in \M^\sC(A_1)$ is of the form $L X$ for $X\in \M^{\Ens}(A_1)$. For any $M$, we have its canonical presentation \cite[(5) p. 153]{mcl}
\begin{equation}\label{eq10}
(L\Omega)^2 M\rightrightarrows L\Omega M\to M 
\end{equation}
whose image by $\Omega$ is a split coequaliser (\emph{loc. cit.}). Given the hypothesis that $\Omega$ creates such coequalisers, \eqref{eq10} is a coequaliser. Since pull-backs are cocontinuous, as well as direct images (again, as left adjoints), \eqref{eq10} remains a coequaliser after applying the functors $(a_2)_*(a ^{12})^*$ and $T^a $. Finally, a coequaliser of isomorphisms is an isomorphism.
\end{proof}

\begin{exs}[for $\sC$] Varieties (category of groups, abelian groups, rings\dots): \cite[VI.8, Th. 1]{mcl}.
\end{exs}

\begin{ex}[for $\sA$] \label{ex3} The category with one object $\underline{G}$ associated with a group $G$: then $\A$ is the category of $G$-sets. Let us take for $\sC$ the category of $R$-modules where $R$ is a commutative ring. If $A\in \A$ is $G$-transitive, $\int A$ is a connected groupoid, which is equivalent to $\underline{H}$ for the stabilizer $H$ of any element of $A$; thus, $\M(A)$ is equivalent to $\Rep_R(H)$. If $a:A_1\to A_0$ is the morphism of $\A$ defined by an inclusion $K\subset H\subset G$ ($A_1=G/K$, $A_0=G/H$), then $ a^*$ is restriction from $H$ to $K$ and $a_*$ is induction $V\mapsto RH\otimes_{RK} V$. From Proposition \ref{p3}, we thus recover conceptually the Mackey formula of \cite[7.3, Prop. 22]{serre}, proven ``by hand'' in \emph{loc. cit.}
\end{ex}

\end{document}